\documentclass[a4paper]{amsart}
\title[{\tiny Stability conditions on some product threefolds}]{Stability conditions on product threefolds of projective spaces and Abelian varieties}
\date{}
\author{Naoki Koseki}

\makeatletter
 
  \@addtoreset{equation}{section}
 \makeatother

\usepackage{amsmath, amssymb, amsthm, amscd}
\usepackage[frame,cmtip,curve,arrow,matrix,line,graph]{xy}

\theoremstyle{plain}
\newtheorem{thm}{Theorem}[section]
\newtheorem{prop}[thm]{Proposition}
\newtheorem{lem}[thm]{Lemma}

\newtheorem*{thm*}{Theorem}

\theoremstyle{definition}
\newtheorem{defin}[thm]{Definition}
\newtheorem{conj}[thm]{Conjecture}
\newtheorem*{NaC}{Notation and Convention}
\newtheorem*{ACK}{Acknowledgement}

\theoremstyle{remark}
\newtheorem{rmk}[thm]{Remark}

\newtheorem{ex}[thm]{Example}

\DeclareMathOperator{\ch}{ch}

\DeclareMathOperator{\id}{id}

\DeclareMathOperator{\Hom}{Hom}
\DeclareMathOperator{\Coh}{Coh}

\DeclareMathOperator{\NS}{NS}

\DeclareMathOperator{\ext}{ext}
\DeclareMathOperator{\Ext}{Ext}
\DeclareMathOperator{\Image}{Image}

\begin{document}
\maketitle

\begin{abstract}
In this paper, we prove 
the original Bogomolov-Gieseker type inequality 
conjecture for $\mathbb{P}^1 \times S, 
\mathbb{P}^2 \times C$, and 
$\mathbb{P}^1 \times \mathbb{P}^1 \times C$, 
where $S$ is an Abelian surface and 
$C$ is an elliptic curve.  
In particular, there exist Bridgeland stability 
conditions on these threefolds. 
\end{abstract}

\setcounter{tocdepth}{1}
\tableofcontents

%======================================================================--
%=========================================================================

\section{Introduction}

\subsection{Motivation and results}
The notion of stability conditions 
on a triangulated category was 
introduced by Bridgeland in his paper \cite{bri07}. 
Bridgeland stability condition is a mathematical 
subject realizing Douglas' $\Pi$-stability 
in string theory \cite{dou01a}, \cite{dou01b}, \cite{dou02}. 
It gives us new points of view 
in various scenes, such as 
birational geometry, counting invariants, Mirror 
symmetry, and so on 
(cf. \cite{abch13}, \cite{bmt14b}, \cite{bm14a}, \cite{bm14c}, \cite{bm14b}, 
\cite{bmt14a}, \cite{tod10}, \cite{tod13a}, \cite{tod13b}, \cite{tod14}). 

Constructing stability conditions 
on the derived category of coherent sheaves 
of a given smooth projective variety $X$ 
is a starting problem for such applications. 
When  $\dim X \leq 2$, the standard construction 
of stability conditions on $D^b(X)$ was given in 
\cite{bri08} and \cite{ab13}. 
In the case when $\dim X=3$, the construction problem of 
stability conditions on $D^b(X)$ is still open in general. 
In the paper \cite{bmt14a}, Bayer, Macr{\`i} and Toda 
proposed a conjectural approach for this problem. 
The problem was reduced to the conjectural 
Bogomolov-Gieseker (BG) type inequality for 
Chern characters (involving the third part of the Chern character) of 
certain semistable objects (called tilt-semistable objects) 
in the derived category. 
It is known that the original BG inequality conjecture holds for 
Abelian threefolds 
(cf. \cite{mp16a}, \cite{mp16b}, \cite{bms14}), 
Fano threefolds of Picard rank one 
(cf. \cite{bmt14a}, \cite{mac14}, \cite{sch14}, \cite{li15}), 
some toric threefolds (cf. \cite{bmsz16}), 
and their \'etale quotients (cf. \cite{mms09}). 

However, counter-examples for 
the original BG inequality conjecture 
were constructed in the case when 
$X$ is the blow-up of a smooth projective threefold 
at a point (cf. \cite{mar16}, \cite{sch16}). 
Furthermore, by using the argument of \cite{mar16}, 
we can show that the BG inequality conjecture does not 
hold even when $X$ is a Calabi-Yau threefold 
containing a plane. 
See Appendix A of this paper. 
Hence we need to modify the inequality in general. 
In this direction, it was shown that some modified 
versions of the BG inequality conjecture hold for every 
Fano threefolds (cf. \cite{bmsz16}). 
On the other hand, it seems still important to 
study for which variety the original BG 
inequality conjecture holds. 
In this paper, we give three new examples 
which satisfy the original BG inequality conjecture: 

\begin{thm}
Let $X$ be $\mathbb{P}^1 \times S$, 
$\mathbb{P}^2 \times C$, or 
$\mathbb{P}^1 \times \mathbb{P}^1 \times C$, 
where $S$ is an Abelian surface and 
$C$ is an elliptic curve. 
Then the original BG inequality conjecture holds for $X$. 
\end{thm}

See Theorem \ref{main thm} for the precise statement. 
In particular, the above theorem implies: 

\begin{thm}
Let $X$ be as above. 
Then there exist Bridgeland stability conditions on $X$. 
\end{thm}

%===========================================================

\subsection{Strategy of the proof of the main theorem}
The idea of proof is borrowed from that of 
\cite{bms14} and \cite{bmsz16}. 
Roughly speaking, they considered 
the Euler characteristic 
$\chi(\mathcal{O}, \underline{m}^{*}E)$ 
of the pull back of a given tilt-semistable object $E$
by the multiplication map 
(resp. toric Frobenius morphism) 
$\underline{m} \colon X \to X$
on an Abelian threefold (resp. a toric threefold) $X$. 
Then by the Riemann-Roch theorem, we know that 
$\chi(\mathcal{O}, \underline{m}^{*}E)$ 
is a polynomial of degree $6$ (resp. $3$) with respect to $m$ and 
its leading coefficient is $\ch_{3}(E)$. 

On the other hand, they showed that 
$\ext^{i}(\mathcal{O}, \underline{m}^{*}E)=O(m^4)$ 
(resp. $O(m^2)$) for even $i$. 
In this way, they got an inequality for the third part of the Chern character, i,e, 
$\ch_{3}(E) \leq 0$. 
To approximate $\ext^{i}(\mathcal{O}, \underline{m}^{*}E)$, 
it was important that $\underline{m}$ is 
\'{e}tale in the case when 
$X$ is an Abelian threefold, while the toric Frobenius splitting 
(Theorem \ref{toric frob}) was essential when $X$ is a toric threefold. 

In this paper, we consider the product of the multiplication map 
on an Abelian variety and the toric Frobenius morphisms on 
the projective spaces. 
Then we approximate $\ext^{i}(\mathcal{O}, \underline{m}^{*}E)$ 
combining the methods in \cite{bms14} and \cite{bmsz16}. 
Note that our approach cannot apply to 
the product threefolds  of an elliptic curve and other toric surfaces 
for a technical reason (see Remark \ref{other toric surface}).

%=========================================================-

\subsection{Plan of the paper}
The paper is organized as follows. 
In Section \ref{prel}, we recall the notion of 
stability conditions. After that, we recall the 
work of \cite{bmt14a} and state our main theorem. 
In Section \ref{prep}, we collect key results 
which we will use in the proof of our main theorem. 
In Section \ref{pf of main thm}, we prove our main theorem. 
In Appendix A, we will show that 
the original BG inequality conjecture for 
a Calabi-Yau threefold containing a plane 
does not hold. 

%=====================================================================

\begin{ACK}
I would like to thank my supervisor Professor Yukinobu Toda. 
He suggested this problem to me and gave various 
comments and advices. 
I would also like to thank Genki Ouchi for useful discussions. 
\end{ACK}

%==============================================================-

\begin{NaC}
In this paper we always work over $\mathbb{C}$. 
We use the following notations: 
\begin{itemize}
\item $\ch^{B}=(\ch_{0}^{B}, \cdots, \ch_{n}^{B}):=e^{-B}.\ch$, 
where $\ch$ denotes the Chern character 
and $B \in H^{2}(X; \mathbb{R})$. 
\item $\ch^{\beta}:=\ch^{\beta H}$, 
where $H$ is an ample divisor and 
$\beta \in \mathbb{R}$. 
\item $H.\ch^{B}:=(H^n.\ch^{B}_{0}, \cdots, H.\ch^{B}_{n-1}, \ch^{B}_{n})$. 
\item $K(\mathcal{A})$ : 
the Grothendieck group of an abelian category $\mathcal{A}$. 
\item $\hom(E, F):=\dim\Hom(E, F)$. 
\item $\ext^i(E, F):=\dim\Ext^i(E, F)$. 
\item $D^b(X):=D^b(\Coh(X))$ : 
the bounded derived category of coherent sheaves 
on a smooth projective variety $X$. 
\end{itemize}
\end{NaC}

%=====================================================-
%=======================================================

\section{Preliminaries}
\label{prel}

\subsection{Bridgeland stability condition}
In this subsection, we recall the definition of 
stability conditions due to Bridgeland \cite{bri07}. 
First we define the notion of stability functions: 

\begin{defin}
\begin{enumerate}
\item Let $\mathcal{A}$ be an abelian category. 
A {\it stability function} on $\mathcal{A}$ is 
a group homomorphism 
$Z \colon K(\mathcal{A}) \to \mathbb{C}$ 
such that 
\[
Z(\mathcal{A} \setminus \{0\}) \subset \mathcal{H}. 
\]
Here $\mathcal{H}:=\mathbb{H} \cup \mathbb{R}_{<0}$ 
is the union of upper half plane and negative real line. 
\item Let $Z$ be a stability function on an abelian category 
$\mathcal{A}$. For $E \in \mathcal{A}$, define 
$\phi_{Z}(E):=\frac{1}{\pi}\arg Z(E) \in (0, 1]$. 
Then $E$ is $Z${\it -semistable} (resp. {\it stable}) if 
for every proper non-zero subobject $F \subset E$, 
\[
\phi_{Z}(F) \leq (\mbox{resp.}<) \phi_{Z}(E). 
\]

\item $Z$ satisfies the 
{\it Harder-Narashimhan property} 
({\it HN property}) if 
the following property holds: 
For every non-zero object $E \in \mathcal{A}$, 
there exists a finite filtration 
\[
0=E_{0} \subset E_{1} \subset \cdots \subset E_{n-1} \subset E_{n}=E
\]
such that $F_{i}:= E_{i}/E_{i-1}$ is $Z$-semistable 
for every $i=1, \cdots, n$ and 
\[
\phi_{Z}(E_{1}) > \cdots >\phi_{Z}(E_{n}). 
\]
\end{enumerate}
\end{defin}

Now we can define the notion of stability conditions 
on a triangulated category. 

\begin{defin}
Let $\mathcal{D}$ be a triangulated category. 
A {\it stability condition} $\sigma=(Z, \mathcal{A})$ 
on $\mathcal{D}$ is a pair consisting of 
the heart of a bounded t-structure $\mathcal{A}$ on $\mathcal{D}$
and a stability function $Z$ on $\mathcal{A}$ 
(called central charge) satisfying the HN-property. 
\end{defin}

\subsection{Stability conditions on smooth projective varieties.}
In this subsection, we recall the works about 
the stability conditions on smooth projective varieties. 
Let $X$ be a smooth projective variety, 
$\omega$ an ample $\mathbb{R}$-divisor on $X$, 
and $B$ any $\mathbb{R}$-divisor on $X$. 
Conjecturally, a group homomorphism 
\[
Z_{\omega, B}:=-\int_{X}e^{-i\omega}.\ch^B 
\colon K(X) \to \mathbb{C} 
\]
becomes the central charge of some stability condition on $D^b(X)$ 
(cf. \cite{bmt14a}, Conjecture 2.1.2). 

When $\dim X=1$, the pair $(Z_{\omega, B}, \Coh(X))$ 
is a stability condition on $X$ and 
this coincides with the Mumford's slope stability. 

However in $\dim X \geq 2$, we need a more 
complicated construction of the wanted heart as follows. 
Let us define the slope function on $\Coh(X)$ as 
\[
\mu_{\omega, B} \colon \Coh(X) \to (-\infty, +\infty], \ 
E \mapsto \frac{\omega^{n-1}.\ch_{1}^{B}(E)}{\ch_{0}^B(E)}, 
\]
where $n=\dim X$. 
Define subcategories of $\Coh(X)$ as follows: 
\begin{align*}
&\mathcal{T}_{\omega, B}
:= \left\langle T \in \Coh(X): 
T \mbox{ is }\mu_{\omega, B} 
\mbox{-semistable with } 
\mu_{\omega, B}(T)>0 
\right\rangle, \\ 
&\mathcal{F}_{\omega, B}
:=\left\langle F \in \Coh(X): 
F \mbox{ is } \mu_{\omega, B} 
\mbox{-semistable  with }
\mu_{\omega, B}(F) \leq 0 
\right\rangle. 
\end{align*}

Here, we denote by $\langle S \rangle$ 
the extension closure of a set of objects 
$S \subset \Coh(X)$. 
Due to the HN-property of 
$\mu$-stability, the pair 
$(\mathcal{T}_{\omega, B}, \mathcal{F}_{\omega, B})$ 
is a torsion pair on $\Coh(X)$ 
in the sence of \cite{hrs96}. 
Then we can construct a new heart, 
called the tilting heart of $\Coh(X)$ 
with respect to the torsion pair: 
\[ 
\Coh^{\omega, B}(X)
:=
\left\langle
\mathcal{F}_{\omega, B}[1], 
\mathcal{T}_{\omega, B}
\right\rangle 
=\left\{E \in \Coh(X) : 
\begin{array}{ll}
\mathcal{H}^{-1}(E) \in \mathcal{F}_{\omega, B}, 
\mathcal{H}^{0}(E) \in \mathcal{T}_{\omega, B}, \\
\mathcal{H}^{i}(E)=0 \mbox{ for } i \neq -1, 0 
\end{array}
\right\}. 
\] 
In $\dim X=2$, $\Coh^{\omega, B}(X)$ is the required heart: 

\begin{thm}[(\cite{bri08}, \cite{ab13})]
Let $\dim X=2$. Then the pair 
$(Z_{\omega, B}, \Coh^{\omega, B}(X))$ 
is a stability condition on $X$. 
\end{thm}

In $\dim X=3$, Bayer, Macr{\`i} and Toda provided the conjectural approach 
to construct the required heart (\cite{bmt14a}). 
The idea is to tilt the heart $\Coh^{\omega, B}(X)$ 
once again by using a new slope function. 
Let us recall the work \cite{bmt14a} of 
Bayer, Macr{\`i} and Toda. 
In the followings, assume that $\dim X=3$. 
Let $H$ be an ample divisor on $X$ and 
let $\omega:=\alpha\sqrt{3}H, B:=\beta H$ \ 
$(\alpha, \beta \in \mathbb{R}, \alpha>0)$. 
Define a slope function on 
$\Coh^{\beta}(X):=\Coh^{\omega, B}(X)$
as follows: 
\[
\nu_{\alpha, \beta}=\nu_{\omega, B} 
\colon \Coh^{\beta}(X) \to (-\infty, +\infty], \ 
E \mapsto 
\frac{H.\ch_{2}^{\beta}(E)-\frac{1}{2}\alpha^2H^3.\ch_{0}^{\beta}(E)}
{H^2.\ch_{1}^{\beta}(E)}. 
\]

Then we can define the notion of 
$\nu_{\alpha, \beta}$-{\it stability} (or {\it tilt-stability}) 
as similar to the $\mu_{\omega, B}$-stability 
on $\Coh(X)$. 
Using the tilt-stability, 
the torsion pair 
$(
\mathcal{T}^{'}_{\alpha, \beta}, 
\mathcal{F}^{'}_{\alpha, \beta} 
)$ 
on $\Coh^{\beta}(X)$ is 
also defined 
similarly to 
$(
\mathcal{T}_{\omega, B}, 
\mathcal{F}_{\omega, B} 
)$ 
on $\Coh(X)$. 
Bayer, Macr{\`i} and Toda 
 conjectured the following BG type inequality 
for $\nu_{\alpha, \beta}$-semistable objects: 

\begin{conj}[(\cite{bmt14a})]
\label{bg conj}
Let $E \in \Coh^{\beta}(X)$ be a 
$\nu_{\alpha, \beta}$-semistable object 
with $\nu_{\alpha, \beta}(E)=0$. Then we have 
\[
\ch_{3}^{\beta}(E) \leq \frac{1}{6}\alpha^2H^2.\ch_{1}^{\beta}(E). 
\]
\end{conj}

Moreover, they showed that the above 
inequality implies the existence of 
a stability condition with the central 
charge $Z_{\alpha, \beta}:=Z_{\omega, B}$. 
Let $\mathcal{A}_{\alpha, \beta}$ be a 
tilting heart of $\Coh^{\beta}(X)$ 
with respect to $\nu_{\alpha, \beta}$-stability, i,e, 
\[
\mathcal{A}_{\alpha, \beta}:=
\left\langle
\mathcal{F}^{'}_{\alpha, \beta}[1], 
\mathcal{T}^{'}_{\alpha, \beta}
\right\rangle. 
\] 

\begin{thm}[(\cite{bmt14a})]
\label{stab condi}
Assume that Conjecture \ref{bg conj} holds. 
Then the pair $(Z_{\alpha, \beta}, \mathcal{A}_{\alpha, \beta})$ 
is a stability condition on $X$. 
\end{thm}

Hence the construction problem of 
stability conditions on $X$ is reduced to 
Conjecture \ref{bg conj}. 
The main theorem of this paper is the following. 

\begin{thm}
\label{main thm}
Let $X$ be $\mathbb{P}^1 \times S, 
\mathbb{P}^2 \times C$, or 
$\mathbb{P}^1 \times \mathbb{P}^1 \times C$, 
where $S$ is an Abelian surface and 
$C$ is an elliptic curve. 
Then for every ample divisor $H$ on $X$, 
$\alpha >0$, and $\beta \in \mathbb{R}$, 
Conjecture \ref{bg conj} holds. 
\end{thm}

\begin{rmk}
In \cite{mar16} and \cite{sch16}, counter-examples 
for Conjecture \ref{bg conj} were obtained when 
$X$ is the blow-up of a smooth projective threefold 
at a point. 
Furthermore, there exists a counter-example even when 
$X$ is a Calabi-Yau threefold containing a plane. 
For the latter, see the appendix of this paper. 
\end{rmk}

%============================================================--

\subsection{Reduction Theorem.} 
In this subsection, we recall the further reduction 
of Conjecture \ref{bg conj} due to \cite{bms14}. 
First we recall the notion of $\bar{\beta}$-stability. 

\begin{defin}
Let $E \in \Coh^{\beta}(X)$ 
be a $\nu_{\alpha, \beta}$-semistable object. 
\begin{enumerate}
\item We define 
\[
\bar{\beta}(E):=
\begin{cases}
\frac{H^2.\ch_{1}(E)-\sqrt{\overline{\Delta}_{H}(E)}}
{H^3.\ch_{0}(E)} & (\ch_{0}(E) \neq 0) \\
\frac{H.\ch_{2}(E)}{H^2.\ch_{1}(E)} & (\ch_{0}(E)=0), 
\end{cases}
\]
where 
\[
\overline{\Delta}_{H}(E)
:=(H^2.\ch_{1}(E))^2-2(H^3.\ch_{0}(E))(H.\ch_{2}(E)). 
\]

\item $E$ is $\bar{\beta}$-{\it semistable} (resp. {\it stable}) if 
there exists an open neighborhood $V$ of $(0, \bar{\beta}(E))$ 
in $(\alpha, \beta)$-plane such that for every 
$(\alpha, \beta) \in V$ with $\alpha >0$, 
$E$ is $\nu_{\alpha, \beta}$-semistable 
(resp. stable). 
\end{enumerate}
\end{defin}

\begin{rmk}
In \cite{bmt14a}, it was shown that 
$\overline{\Delta}_{H}(E)$ is non-negative 
for every $\nu_{\alpha, \beta}$-semistable object $E$. 
\end{rmk}

Then Conjecture \ref{bg conj} is reduced 
as follows: 

\begin{thm}[(\cite{bms14}, Theorem 5.4)]
\label{reduction thm}
Assume that for every $\bar{\beta}$-stable object $E$
with $\ch_{0}(E) \geq 0$ and  $\bar{\beta}(E) \in [0, 1)$, 
we have 
\[
\ch_{3}^{\bar{\beta}(E)}(E) \leq 0. 
\]
Then Conjecture \ref{bg conj} holds for 
every $\alpha, \beta$. 
\end{thm}

%%====================================================================-
%=========================================================================

\section{Preparation for the Main Theorem} 
\label{prep}

In this section, we collect key results which 
we will use in the proof of our main theorem. 
The first one is about the toric Frobenius push forward
of line bundles: 

\begin{thm}[(\cite{tho00})]
\label{toric frob}
Let $Y$ be a smooth projective toric variety, 
let $\{D_{\rho}\}_{\rho}$ be 
the torus invariant divisors. 
For $m \in \mathbb{Z}_{>0}$, denote 
the toric Frobenius morphism by 
$\underline{m} \colon Y \to Y$. 
Then for every divisor $D$ on $Y$, 
we have 
\[
\underline{m}_{*}\mathcal{O}
\left(D
\right)
=\bigoplus_{j} L_{j}^{*\oplus \eta_{j}}, 
\]
where 
\[
L_{j}:=\mathcal{O}
\left(\frac{1}{m}
\left(-D+\sum_{\rho}a_{\rho}D_{\rho}
\right)
\right). 
\]
Here, integers $0 \leq a_{\rho} \leq m-1$ move so that 
$L_{j}$ becomes an integral divisor 
and $\eta_{j}$ counts the multiplicity of 
$\{a_{\rho}\}$ which defines the same $L_{j}$. 
\end{thm}

\begin{rmk}
\label{rmk toric frob}
Let $Y=\mathbb{P}^n$ be a projective space. 
Let $a_{\rho}$ be as in Theorem \ref{toric frob}. 
Then we have 
\[
-K_{Y}-\frac{\sum_{\rho}a_{\rho}D_{\rho}}{m}
= \sum_{\rho}D_{\rho}-\frac{\sum_{\rho}a_{\rho}D_{\rho}}{m}
\geq \frac{1}{m}\sum_{\rho}D_{\rho}
\]
and hence 
$-K_{Y}-\frac{\sum_{\rho}a_{\rho}D_{\rho}}{m}$ 
is ample on $Y$. 
This fact will be used in 
Section \ref{pf of main thm}. 
\end{rmk}

The next one is about the preservation of tilt-stability 
under the pull back by finite \'etale morphisms: 

\begin{prop}[(\cite{bms14}, Proposition 6.1)]
\label{etale pull back}
Let $f \colon Y \to X$ be a finite 
\'etale surjective morphism 
between smooth projective threefolds. 
Let $\omega$ be an ample $\mathbb{R}$-divisor on $X$, 
$B$ an $\mathbb{R}$-divisor on $X$. 
Let $E \in D^b(X)$. Then 
\begin{enumerate}
\item $\nu_{f^{*}\omega, f^{*}B}(f^{*}E)=\nu_{\omega, B}(E)$. 

\item $f^{*}E \in \Coh^{f^{*}\omega, f^{*}B}(Y)$ 
if and only if  
$E \in \Coh^{\omega, B}(X)$. 

\item $f^{*}E$ is $\nu_{f^{*}\omega, f^{*}B}$-semistable 
(resp. stable) if and only if 
$E$ is $\nu_{\omega, B}$-semistable (resp. stable). 
\end{enumerate}
\end{prop}

\begin{ex}
Let $A$ be an Aberian variety of $\dim \leq 2$, 
let $X=Y \times A$ be a product threefold. 
Let $\underline{m} \colon A \to A$ be 
a multiplication map $(m \in \mathbb{Z}_{>0})$. 
Then $\id_{Y} \times \underline{m} \colon X \to X$ 
is a finite \'etale surjective morphism. 
Hence we can apply the above proposition 
to $\id_{Y} \times \underline{m}$. 
\end{ex}

The third one is about the tilt-stability of line bundles: 

\begin{lem}[(\cite{bms14}, Corollary 3.11)]
\label{tilt-stab line bdl}
Let $X$ be a smooth projective threefold, 
$H$ an ample divisor on $X$. 
Assume that for every effective divisor $D$ on $X$, 
we have $H.D^2 \geq 0$. 
Then for every line bundle $L$ on $X$, 
$\alpha>0$, and $\beta \in \mathbb{R}$, 
$L$ or $L[1]$ is $\nu_{\alpha, \beta}$-stable. 
\end{lem}

\begin{ex}
\label{ex of tilt-stab line bdl}
\begin{enumerate}
\item Let $C$ be an elliptic curve, $S$ an Abelian surface. 
Let $X$ be $\mathbb{P}^1 \times S, \mathbb{P}^2 \times C$, 
or $\left(\mathbb{P}^1 \times \mathbb{P}^1 \right) \times C$. 
Then the assumption of the above lemma holds 
for  every ample divisor on $X$, 
since there are no negative divisors on 
projective spaces or Abelian varieties. 

\item Let $Y$ be any smooth projective toric 
surface other than $\mathbb{P}^2, 
\mathbb{P}^1 \times \mathbb{P}^1$. 
Let $X=Y \times C$. 
Since there exists a negative curve on $Y$, 
the assumption in the above lemma does not 
hold for any ample divisor on $X$. 
\end{enumerate}
\end{ex}

\begin{rmk}
\label{other toric surface}
The tilt-stability of line bundles 
is crucial in our proof of the main theorem. 
Hence our approach can not apply to 
threefolds in (ii) of Example \ref{ex of tilt-stab line bdl}. 
\end{rmk}

The last one is about the approximation of dimensions 
of certain $\Ext$-groups due to \cite{bms14}. 

\begin{prop}[(\cite{bms14})]
\label{approx ext}
Let $C$ be an elliptic curve, $S$ an Abelian surface. 
Let $X$ be $\mathbb{P}^1 \times S, \mathbb{P}^2 \times C$, 
or $\mathbb{P}^1 \times \mathbb{P}^1 \times C$. 
Let $f^{(m^2, m)}:=\underline{m^2} \times \underline{m} 
\colon X \to X$ 
be the product of the toric Frobenius morphism and 
the multiplication map. 
Let $E \in D^b(X)$ be a two term complex 
concentrated in degree $-1$ and $0$. 
\begin{enumerate}
\item If there exists an ample divisor $H'$ on $X$ 
such that 
\[
\hom
\left(
\mathcal{O}(H'), f^{(m^2, m)*}E
\right)
=0, 
\]
then 
\[
\hom
\left(
\mathcal{O}, f^{(m^2, m)*}E
\right)
=O(m^4). 
\]

\item If there exists an ample divisor $H'$ on $X$ 
such that 
\[
\ext^2
\left(
\mathcal{O}(-H'), f^{(m^2, m)*}E
\right)
=0, 
\]
then 
\[
\ext^2
\left(
\mathcal{O}, f^{(m^2, m)*}E
\right)
=O(m^4). 
\]
\end{enumerate}
\end{prop}

\begin{proof}
Summarizing the arguments of 
Section 7 in \cite{bms14}, 
we get the result. 
\end{proof}

%=================================================================
%==================================================================

\section{Proof of the Main Theorem}
\label{pf of main thm}

In this section, we prove our main theorem, Theorem \ref{main thm}. 
Let $C$ be an elliptic curve and 
$S$ an Abelian surface. 
Let $X=Y \times Z$, where 
$(Y, Z)=(\mathbb{P}^1, S), (\mathbb{P}^2, C)
, (\mathbb{P}^1 \times \mathbb{P}^1, C)$. 
Let $H$ be an ample divisor on $X$. 
Then $H$ can be written as $H=h+f$, 
where $h, f$ are the pull back of some 
ample divisors on $Y, Z$, respectively. 
For integers $a, b \in \mathbb{Z}_{\geq 0}$, 
let $f^{(a, b)}:=\underline{a} \times \underline{b}$ 
be the product of the toric Frobenius morphism 
$\underline{a}$ on $Y$ 
and the multiplication map 
$\underline{b}$ on $Z$. 
Furthermore, let us denote by 
$D_{\rho} \in \NS(X)$ the pull backs of 
the torus invariant divisors on Y. 

\begin{rmk}
\label{act on cohomology}
Let $m \in \mathbb{Z}_{>0}$. 
Note that $f^{(m^2, m)*}$ acts on the even cohomology 
as follows: 
\[
\bigoplus_{i=0}^{3}H^{2i}(X) 
\ni (x, y, z, w) \mapsto (x, m^2y, m^4z, m^6w)
\in \bigoplus_{i=0}^{3}H^{2i}(X). 
\]

We will use this property in the followings. 
\end{rmk}

Let $E$ be a $\bar{\beta}$-stable object 
with $\ch_{0}(E) \geq 0$ and 
$\bar{\beta}:=\bar{\beta}(E) \in [0, 1)$. 
To prove Theorem \ref{main thm}, it is 
enough to show that 
$\ch_{3}^{\bar{\beta}}(E) \leq 0$ 
by Theorem \ref{reduction thm}. 
We prove it in the following three subsections. 
We start with  two easy lemmas 
which we will frequently use in the 
followings. 

\begin{lem}
\label{adjoint}
For every $E, F \in D^b(X)$, we have 
\[
\Hom\left(
\omega_{X}^{*} \otimes f^{(a, b)}_{*}\left(
E \otimes \omega_{X}
\right), 
F
\right) 
\cong \Hom\left(
E, f^{(a, b)*}F
\right). 
\]
\end{lem}
\begin{proof}
Use Serre duality and the adjointness 
between $f^{(a, b)*}$ and $f^{(a, b)}_{*}$. 
Note that we do not need to take 
derived functors since 
$f^{(a, b)}$ is finite and flat. 
\end{proof}

\begin{lem}
\label{ext vanish}
Let $E \in \Coh^{\bar{\beta}}(X)$ be 
a $\bar{\beta}$-stable object with 
$\ch_{0}(E) \geq 0$ and $\bar{\beta}=\bar{\beta}(E) \in [0, 1)$, 
$L$ a line bundle on $X$. 
\begin{enumerate}
\item If $\ch_{1}^{\bar{\beta}}(L)$ is ample, then 
\[
\hom\left(L, E
\right)=0. 
\]

\item If $\ch_{1}^{\bar{\beta}}(L)$ is anti-ample, then 
\[
\hom\left(E, L[1]
\right)=0. 
\]
\end{enumerate}
\end{lem}
\begin{proof}
We only prove the first statement. 
The second one also follows from 
the similar computation. 
Since $\ch_{1}^{\bar{\beta}}(L)$ is ample, 
we have 
\[
H^2.\ch_{1}^{\bar{\beta}}(L)>0, \ 
H.\ch_{1}^{\bar{\beta}}(L)^2>0. 
\]
By the first inequality, we have 
$L \in \Coh^{\bar{\beta}}(X)$. 
Moreover, by Proposition \ref{tilt-stab line bdl}, 
$L$ is tilt-stable near $(0, \bar{\beta})$. 
On the other hand, the second inequality implies 
\[
\begin{split}
\lim_{\alpha \to +0}\nu_{\alpha, \bar{\beta}}(L)
&=\frac{H.\ch_{2}^{\bar{\beta}}(L)}{H^2.\ch_{1}^{\bar{\beta}}(L)} \\ 
&=\frac{\frac{1}{2}H.\ch_{1}^{\bar{\beta}}(L)^2}{H^2.\ch_{1}^{\bar{\beta}}(L)} \\ 
&>0=\lim_{\alpha \to 0}\nu_{\alpha, \bar{\beta}}(E). 
\end{split}
\]
Hence by the tilt-stability of $L$ and $E$, we have 
\[
\hom
\left(L, E
\right)=0. 
\]
\end{proof}

%==================================================================---

\subsection{Integral case}

Assume that $\bar{\beta}=0$. 
Let us consider $\chi(\mathcal{O}, f^{(m^2, m)*}E)$. 
By the Riemann-Roch Theorem 
and Remark \ref{act on cohomology}, we have 
\[
\chi\left(
\mathcal{O}, f^{(m^2, m)*}E
\right)
=m^6\ch_{3}(E)+O(m^4). 
\]
On the other hand, 
\begin{equation}
\label{integral euler}
\chi\left(
\mathcal{O}, f^{(m^2, m)*}E
\right) \leq 
\hom\left(
\mathcal{O}, f^{(m^2, m)*}E
\right) 
+\ext^2\left(\mathcal{O}, f^{(m^2, m)*}E
\right)
\end{equation}
since $E$ is a two term complex concentrated 
in degree $-1$ and $0$. 

By Proposition \ref{approx ext}, 
the following two lemmas show that the RHS of 
the inequality (\ref{integral euler}) is of order $m^4$. 
Hence we must have $\ch_{3}(E) \leq 0$ 
as required. 
 
\begin{lem}
We have 
\[
\hom\left(
\mathcal{O}\left(
-K_{X}+f
\right), 
f^{(m^2, m)*}E
\right)=0. 
\]
\end{lem}

\begin{proof}
By Theorem \ref{toric frob} and 
Lemma \ref{adjoint}, we have 
\[
\begin{split}
\hom\left(
\mathcal{O}\left(
-K_{X}+f
\right), 
f^{(m^2, m)*}E
\right)
&=\hom\left(
\mathcal{O}\left(
-K_{X}
\right) \otimes 
f^{(m^2, 1)}_{*}\mathcal{O}\left(
-K_{X}+f+K_{X}
\right), 
f^{(1, m)*}E
\right) \\
&=\hom\left(
\mathcal{O}\left(
-K_{X}+f
\right) \otimes
\left(\oplus L_{j}^{*\oplus \eta_{j}}
\right), f^{(1, m)*}E
\right), 
\end{split}
\]
where 
\[
L_{j}=\mathcal{O}
\left(
\frac{1}{m^2}\left(
\sum_{\rho}a_{\rho}D_{\rho}
\right)
\right), \ 
0 \leq a_{\rho} \leq m^2-1. 
\]
As remarked in Remark \ref{rmk toric frob}, 
$\mathcal{O}(-K_{X}) \otimes L_{j}^{*}$ is 
the pull back of an ample line bundle on $Y$. 
Hence $\mathcal{O}(-K_{X}+f) \otimes L_{j}^{*}$ 
is ample on $X$. 
Note that $f^{(1, m)*}E$ is $\bar{\beta}$-stable 
with $\bar{\beta}(f^{(1, m)*}E)=0$ 
(with respect to the polarization $f^{(1, m)*}H$) 
by Proposition \ref{etale pull back}. 
Hence Lemma \ref{ext vanish} implies that 
\[
\hom\left(
\mathcal{O}\left(
-K_{X}+f
\right) 
\otimes L_{j}^{*}, f^{(1, m)*}E
\right)=0. 
\]
Summing up, we conclude that 
\[
\hom\left(
\mathcal{O}\left(
-K_{X}+f
\right), 
f^{(m^2, m)*}E
\right)=0.
\]
\end{proof}

\begin{lem}
We have 
\[
\ext^2\left(
\mathcal{O}\left(
-h-f
\right), 
f^{(m^2, m)*}E
\right)=0. 
\]
\end{lem}

\begin{proof}
By Theorem \ref{toric frob}, Serre duality, and the usual adjoint, 
we have 
\[
\begin{split} 
\ext^2\left(
\mathcal{O}\left(
-h-f
\right), 
f^{(m^2, m)*}E
\right)
&=\hom\left(
f^{(m^2, m)*}E, 
\mathcal{O}\left(
-h-f+K_{X}
\right)[1]
\right) \\
&=\hom\left(
f^{(1, m)*}E, 
\mathcal{O}\left(
-f
\right) 
\otimes f^{(m^2, 1)}_{*}
\mathcal{O}\left(
-h+K_{X}
\right)[1]
\right) \\
&=\hom\left(
f^{(1, m)*}E, 
\mathcal{O}\left(
-f
\right) 
\otimes \left(
\oplus L_{j}^{*\oplus \eta_{j}}
\right)[1]
\right), 
\end{split}
\]
where 
\[
L_{j}=\mathcal{O}\left(
\frac{1}{m^2}\left(
h-K_{X}+\sum_{\rho}a_{\rho}D_{\rho}
\right)
\right), \ 
0 \leq a_{\rho} \leq m^2-1. 
\]
For all $j$, 
$\ch_{1}(\mathcal{O}(-f) \otimes L_{j}^{*})$
is anti-ample. 
Hence by Lemma \ref{ext vanish}, 
\[
\ext^2\left(
\mathcal{O}\left(
-h-f
\right), 
f^{(m^2, m)*}E
\right)=0.
\]
\end{proof}

\subsection{Rational case}

In this subsection, we assume that 
$\bar{\beta}=\frac{p}{q} \in \mathbb{Q}$, 
$q>0$, $p$ and $q$ are coprime. 
We consider 
$\chi\left(\mathcal{O}, 
f^{(m^2, m)*}\left(
f^{(q^2, q)*}E 
\otimes \mathcal{O}\left(
-pqH
\right)
\right)
\right)$. 
By Remark \ref{act on cohomology}, 
we have 

\[
\ch_{3}\left(
f^{(q^2, q)*}E 
\otimes \mathcal{O}\left(
-pqH
\right)
\right)
=q^6\ch_{3}^{\frac{p}{q}}(E) 
\]
and hence 
\[
\begin{split}
m^6q^6\ch_{3}^{\frac{p}{q}}(E)+O(m^4)
&=\chi\left(\mathcal{O}, 
f^{(m^2, m)*}\left(
f^{(q^2, q)*}E 
\otimes \mathcal{O}\left(
-pqH
\right)
\right)
\right)\\
& \leq \hom\left(\mathcal{O}, 
f^{(m^2, m)*}\left(
f^{(q^2, q)*}E 
\otimes \mathcal{O}\left(
-pqH
\right)
\right)
\right) \\
&\quad+\ext^2\left(\mathcal{O}, 
f^{(m^2, m)*}\left(
f^{(q^2, q)*}E 
\otimes \mathcal{O}\left(
-pqH
\right)
\right)
\right). 
\end{split}
\]
As in the previous subsection, 
we will check the assumption in 
Proposition \ref{approx ext}. 

\begin{lem}
We have 
\[
\hom\left(
\mathcal{O}\left(
-K_{X}+f
\right), 
f^{(m^2, m)*}\left(
f^{(q^2, q)*}E 
\otimes \mathcal{O}\left(
-pqH
\right)
\right)
\right)=0.
\]
\end{lem}

\begin{proof}
Using Theorem \ref{toric frob} and 
Lemma \ref{adjoint}, we have 

\[
\begin{split}
&\hom\left(
\mathcal{O}\left(
-K_{X}+f
\right), 
f^{(m^2, m)*}\left(
f^{(q^2, q)*}E 
\otimes \mathcal{O}\left(
-pqH
\right)
\right)
\right) \\
=&\hom\left(
\mathcal{O}\left(
-K_{X}
\right) 
\otimes \mathcal{O}\left(
f
\right) 
\otimes f^{(m^2, 1)}_{*}\mathcal{O}\left(
-K_{X}+K_{X}
\right), 
f^{(q^2, 1)*}f^{(1, mq)*}E 
\otimes f^{(1, m)*}\mathcal{O}\left(
-pqH
\right)
\right) \\
=&\hom\left(
\mathcal{O}\left(
-K_{X}+f
\right) 
\otimes \left(
\oplus L_{j}^{*\oplus \eta_{j}}
\right), 
f^{(q^2, 1)*}f^{(1, mq)*}E 
\otimes \mathcal{O}\left(
-pqh-pqm^2f
\right)
\right) \\
=&\sum_{j}\eta_{j}\hom\left(
\mathcal{O}\left(
pqm^2f+f
\right) 
\otimes \mathcal{O}\left(
-K_{X}+pqh
\right) 
\otimes L_{j}^{*}, 
f^{(q^2, 1)*}f^{(1, mq)*}E
\right) \\
=&\sum_{j}\eta_{j}\hom\left(
\mathcal{O}\left(
pqm^2f+f
\right) 
\otimes \mathcal{O}
\left(
-K_{X}
\right) 
\otimes f^{(q^2, 1)}_{*}\left(
\mathcal{O}\left(
-K_{X}+pqh+K_{X}
\right) 
\otimes L_{j}^{*}
\right), 
f^{(1, mq)*}E
\right) \\
=&\sum_{j}\eta_{j}\hom\left(
\mathcal{O}\left(
pqm^2f+f
\right) 
\otimes \mathcal{O}\left(
-K_{X}
\right) 
\otimes \left(
\oplus R_{kj}^{* \oplus \epsilon_{kj}}
\right), 
f^{(1, mq)*}E
\right), 
\end{split}
\]
where 
\begin{align*}
&L_{j}=\mathcal{O}\left(
\frac{1}{m^2}\sum_{\rho}a_{\rho}D_{\rho}
\right), \ 
0 \leq a_{\rho} \leq m^2-1 \quad \text{and} \\
&R_{kj}=\mathcal{O}\left(
\frac{1}{q^2} \left(
-pqh+\frac{1}{m^2}\sum_{\rho}a_{\rho}D_{\rho}
+\sum_{\rho}b_{\rho}D_{\rho}
\right)
\right), \ 
0 \leq b_{\rho} \leq q^2-1. 
\end{align*}

By Lemma \ref{ext vanish}, it is enough to show that 
\[
\ch_{1}^{
\frac{p}{q}
\left(
h+m^2q^2f
\right)
}
\left(
\mathcal{O}\left(
pqm^2f+f
\right) 
\otimes \mathcal{O}\left(
-K_{X}
\right) 
\otimes R_{kj}^{*}
\right) 
\]
is ample. 
We can compute it as 
\[
\begin{split}
\ch_{1}^{
\frac{p}{q}
\left(
h+m^2q^2f
\right)
}
\left(
\mathcal{O}\left(
pqm^2f+f
\right) 
\otimes \mathcal{O}\left(
-K_{X}
\right) 
\otimes R_{kj}^{*}
\right) 
&= pqm^2f+f-K_{X}+
\frac{pqh}{q^2}-\frac{\sum_{\rho}a_{\rho}D_{\rho}}{m^2q^2} \\
&\quad -\frac{\sum_{\rho}b_{\rho}D_{\rho}}{q^2} 
-\frac{p}{q}\left(
h+m^2q^2f
\right) \\
&=f-K_{X}-\frac{\sum_{\rho}a_{\rho}D_{\rho}}{m^2q^2}
-\frac{\sum_{\rho}b_{\rho}D_{\rho}}{q^2}. 
\end{split}
\]
Since 
$-K_{X}=\sum_{\rho}D_{\rho}$, 
$0 \leq a_{\rho} \leq m^2-1$, 
and $0 \leq b_{\rho} \leq q^2-1$, 
we have 
\[
\begin{split}
f-K_{X}-\frac{\sum_{\rho}a_{\rho}D_{\rho}}{m^2q^2}
-\frac{\sum_{\rho}b_{\rho}D_{\rho}}{q^2} 
&= f+\sum_{\rho}D_{\rho}
-\frac{\sum_{\rho}a_{\rho}D_{\rho}}{m^2q^2}
-\frac{\sum_{\rho}b_{\rho}D_{\rho}}{q^2} \\
&\geq f+\frac{1}{m^2q^2}
\left(
m^2q^2-(m^2-1)-m^2(q^2-1)
\right) \sum_{\rho}D_{\rho} \\
&=f+\frac{1}{m^2q^2}\sum_{\rho}D_{\rho} 
\end{split}
\]
and it is ample on $X$. 
We conclude that 
\[
\hom\left(
\mathcal{O}\left(
-K_{X}+f
\right), 
f^{(m^2, m)*}\left(
f^{(q^2, q)*}E 
\otimes \mathcal{O}\left(
-pqH
\right)
\right)
\right)=0.
\]
\end{proof}

\begin{lem}
We have 
\[
\ext^2\left(
\mathcal{O}\left(
-h-f
\right), 
f^{(m^2, m)*}\left(
f^{(q^2, q)*}E 
\otimes \mathcal{O}\left(
-pqH
\right)
\right)
\right)=0. 
\]
\end{lem}

\begin{proof}
By Serre duality, adjunction, and Theorem \ref{toric frob}, 
we have 
\[
\begin{split}
&\ext^2\left(
\mathcal{O}\left(
-h-f
\right), 
f^{(m^2, m)*}\left(
f^{(q^2, q)*}E 
\otimes \mathcal{O}\left(
-pqH
\right)
\right)
\right) \\
=&\hom\left(
f^{(m^2, m)*}\left(
f^{(q^2, q)*}E 
\otimes \mathcal{O}\left(
-pqH
\right)
\right), 
\mathcal{O}\left(
-f-h+K_{X}
\right)[1]
\right) \\
=&\hom\left(
f^{(1, m)*}\left(
f^{(q^2, q)*}E 
\otimes \mathcal{O}\left(
-pqH
\right)
\right), 
\mathcal{O}\left(
-f
\right) 
\otimes f^{(m^2, 1)}_{*}\mathcal{O}\left(
-h+K_{X}
\right)[1]
\right) \\
=&\hom\left(
f^{(q^2, 1)*}f^{(1, mq)*}E, 
\mathcal{O}\left(
pqh+pqm^2f
\right) 
\otimes \mathcal{O}\left(-f
\right) 
\otimes \left(
\oplus L_{j}^{* \oplus \eta_{j}}
\right)[1]
\right) \\
=&\sum_{j}\eta_{j}\hom\left(
f^{(q^2, 1)*}f^{(1, mq)*}E, 
\mathcal{O}\left(
pqh+pqm^2f
\right) 
\otimes \mathcal{O}\left(
-f
\right) 
\otimes L_{j}^{*}[1]
\right) \\
=&\sum_{j}\eta_{j}\hom\left(
f^{(q^2, 1)*}f^{(1, mq)*}E, 
\mathcal{O}\left(
pqm^2f-f
\right) 
\otimes \mathcal{O}\left(
pqh
\right) 
\otimes L_{j}^{*}[1]
\right) \\
=&\sum_{j}\eta_{j}\hom\left(
f^{(1, mq)*}E, 
\mathcal{O}\left(
pqm^2f-f
\right) 
\otimes f^{(q^2, 1)}_{*}\left(
\mathcal{O}\left(
pqh
\right) 
\otimes L_{j}^{*}
\right)[1]
\right) \\
=&\sum_{j}\eta_{j}\hom\left(
f^{(1, mq)*}E, 
\mathcal{O}\left(
pqm^2f-f
\right) 
\otimes \left(
\oplus R_{k_{j}}^{*\oplus \epsilon_{kj}}
\right)[1]
\right). 
\end{split}
\]
Here, 
\begin{align*}
&L_{j}=\mathcal{O}\left(
\frac{1}{m^2}
\left(
h-K_{X}+\sum a_{\rho}D_{\rho}
\right)
\right), \ 
0 \leq a_{\rho} \leq m^2-1
\quad \mbox{and} \\
&R_{kj}=\mathcal{O}\left(
\frac{1}{q^2}\left(
-pqh
+\frac{h-K_{X}+\sum a_{\rho}D_{\rho}}{m^2}
+\sum b_{\rho}D_{\rho}
\right)
\right), \ 
0 \leq b_{\rho} \leq q^2-1. 
\end{align*}

As before, it is enough to show that 
\[
\ch_{1}^{
\frac{p}{q}
\left(
h+m^2q^2f
\right)
}
\left(
\mathcal{O}\left(
pqm^2f-f
\right) 
\otimes R_{kj}^{*}
\right)
\]
is anti-ample. 
Straightforward computation yields that 
\[
\begin{split}
\ch_{1}^{
\frac{p}{q}
\left(
h+m^2q^2f
\right)
}
\left(
\mathcal{O}\left(
pqm^2f-f
\right) 
\otimes R_{kj}^{*}
\right)
&=pqm^2f-f 
+\frac{p}{q}h 
-\frac{h-K_{X}+\sum a_{\rho}D_{\rho}}{m^2q^2} 
-\frac{\sum b_{\rho}D_{\rho}}{q^2} \\ 
&\quad-\frac{p}{q}\left(
h+m^2q^2f
\right) \\
&=-f-\frac{h-K_{X}+\sum a_{\rho}D_{\rho}}{m^2q^2} 
-\frac{\sum b_{\rho}D_{\rho}}{q^2}. 
\end{split}
\]
This is anti-ample on $X$. 
Hence we get the required result. 
\end{proof}

%================================================================

\subsection{Irrational case} 

Assume that $\bar{\beta}$ is irrational. 
Define 
\[
V_{\epsilon}:=\{
(\alpha, \beta) : 0<\alpha<\epsilon, 
\bar{\beta}-\epsilon < \beta <\bar{\beta}+\epsilon
\}. 
\]
Take $\epsilon >0$ small enough so that 
for every $(\alpha, \beta) \in V_{\epsilon}$, 
$E$ is $\nu_{\alpha, \beta}$-stable. 
By the Dirichlet approximation theorem, 
we can take a sequence 
$\left\{
\beta_{n}=\frac{p_{n}}{q_{n}}
\right\}$ 
of rational numbers 
such that 
\[
\left|
\bar{\beta}-\beta_{n}
\right| 
<\frac{1}{q_{n}^2}<\epsilon
\]
and $q_{n} \to +\infty$ as $n \to +\infty$. 
We compute 
$\chi\left(
\mathcal{O}, 
f^{(q_{n}^2, q_{n})*}E 
\otimes \mathcal{O}\left(
-p_{n}q_{n}H
\right)
\right)$. 
As before, 
\[
\begin{split}
&q_{n}^6\ch_{3}^{\bar{\beta}}(E)+O(q_{n}^4) \\
\leq& q_{n}^6\ch_{3}^{\beta_{n}}(E)+O(q_{n}^4) \\
\leq&\chi\left(
\mathcal{O}, 
f^{(q_{n}^2, q_{n})*}E 
\otimes \mathcal{O}\left(
-p_{n}q_{n}H
\right)
\right) \\ 
\leq& \hom\left(
\mathcal{O}, 
f^{(q_{n}^2, q_{n})*}E 
\otimes \mathcal{O}\left(
-p_{n}q_{n}H
\right)
\right) 
+\ext^2\left(
\mathcal{O}, 
f^{(q_{n}^2, q_{n})*}E 
\otimes \mathcal{O}\left(
-p_{n}q_{n}H
\right)
\right). 
\end{split}
\]

We will show that the last line of the 
above inequalities is of order $q^4_{n}$. 

\begin{lem}
Let $u, v \in \mathbb{Z}_{>0}$ such that 
$(u-2)h+K_{X}$ is effective and $v>2$. Then 
\[
\hom\left(
\mathcal{O}\left(
uh+vf
\right), 
f^{(q_{n}^2, q_{n})*}E 
\otimes \mathcal{O}\left(
-p_{n}q_{n}H
\right)
\right)=0. 
\]
\end{lem}
\begin{proof}
As in the rational case, we have 
\[
\begin{split}
&\hom\left(
\mathcal{O}\left(
uh+vf
\right), 
f^{(q_{n}^2, q_{n})*}E 
\otimes \mathcal{O}\left(
-p_{n}q_{n}H
\right)
\right) \\
=&\hom\left(
\mathcal{O}\left(
\left(
p_{n}q_{n}+v
\right)f
\right) 
\otimes \mathcal{O}\left(
-K_{X}
\right) 
\otimes \left(
\oplus L_{j}^{*\oplus \eta_{j}}
\right), 
f^{(1, q_{n})*}E
\right), 
\end{split}
\]
where 
\[
L_{j}=\mathcal{O}\left(
\frac{1}{q_{n}^2}\left(
-\left(
p_{n}q_{n}+u
\right)h
-K_{X}
+\sum a_{\rho}D_{\rho}
\right)
\right), \ 
0 \leq a_{\rho} \leq q_{n}^2-1. 
\]

Let $M_{j}:=\mathcal{O}\left(
\left(
p_{n}q_{n}+v
\right)f
\right) 
\otimes \mathcal{O}\left(
-K_{X}
\right) 
\otimes L_{j}^{*}$, 
$H^{(n)}:=f^{(1, q_{n})*}H$. 
For a while, assume that 
$\ch^{\beta_{n}H^{(n)}}_{1}(M_{j}) 
-2\frac{1}{q_{n}^2}H^{(n)}$ is ample. 
Then we can compute as  
\[
\begin{split}
H^{(n)2}.\ch^{\bar{\beta}H^{(n)}}_{1}(M_{j})
&=H^{(n)2}.\left(
\ch^{\beta_{n}H^{(n)}}_{1}\left(
M_{j}
\right)
+\left(
\beta_{n}-\bar{\beta}
\right)H^{(n)}
\right) \\ 
&>H^{(n)2}.\left(
\ch^{\beta_{n}H^{(n)}}_{1}\left(
M_{j}
\right)
-\frac{1}{q_{n}^2}H^{(n)}
\right) \\ 
&>0 
\end{split}
\]
and 
\[
\begin{split}
H^{(n)}.\ch^{\bar{\beta}H^{(n)}}_{1}(M_{j})^2 
&=H^{(n)}.\left(
\ch^{\beta_{n}H^{(n)}}_{1}(M_{j})
+\left(
\beta_{n}-\bar{\beta}
\right)H^{(n)}
\right)^2 \\
&>H^{(n)}.\ch^{\beta_{n}H^{(n)}}_{1}(M_{j})
.\left(
\ch^{\beta_{n}H^{(n)}}_{1}(M_{j}) 
-2\frac{1}{q_{n}^2}H^{(n)}
\right) \\
&>0, 
\end{split}
\]
which imply 
\[
\hom\left(M_{j}, 
f^{(1, q_{n})*}E
\right)=0
\]
by the proof of Lemma \ref{ext vanish}. 
Hence it is enough to show that 
$\ch^{\beta_{n}H^{(n)}}_{1}(M_{j}) 
-2\frac{1}{q_{n}^2}H^{(n)}$ is ample. 
As in the rational case, we have 
\[
\ch^{\beta_{n}H^{(n)}}_{1}(M_{j})
=vf-K_{X}
+\frac{uh+K_{X}-\sum a_{\rho}D_{\rho}}{q_{n}^2} 
\]
and hence 
\[
\begin{split}
\ch^{\beta_{n}H^{(n)}}_{1}(M_{j}) -2\frac{1}{q_{n}^2}H^{(n)}
&=vf-K_{X}
+\frac{uh+K_{X}-\sum a_{\rho}D_{\rho}}{q_{n}^2} 
-2\frac{h+q_{n}^2f}{q_{n}^2} \\
&=(v-2)f-K_{X}
-\frac{\sum a_{\rho}D_{\rho}}{q_{n}^2}
+\frac{(u-2)h+K_{X}}{q_{n}^2}. 
\end{split}
\]
As observed in Remark \ref{rmk toric frob}, 
\[
-K_{X}
-\frac{\sum a_{\rho}D_{\rho}}{q_{n}^2}
\]
is the pull back of an ample divisor on $Y$. 
Hence if we take $u, v$ so that 
$v>2$ and $(u-2)h+K_{X}$ is effective, 
then 
\[
\ch^{\beta_{n}H^{(n)}}_{1}(M_{j}) -2\frac{1}{q_{n}^2}H^{(n)}
\]
is ample on $X$. 
Note that these conditions does not depend on $n$. 
\end{proof}

\begin{lem}
Let $u, v \in \mathbb{Z}_{>0}$, $u, v>2$. Then   
\[
\ext^2\left(
\mathcal{O}\left(
-uh-vf
\right), 
f^{(q_{n}^2, q_{n})*}E 
\otimes \mathcal{O}\left(
-p_{n}q_{n}H
\right)
\right)=0. 
\]
\end{lem}

\begin{proof}
As in the rational case, 
\[
\begin{split}
&\ext^2\left(
\mathcal{O}\left(
-uh-vf
\right), 
f^{(q_{n}^2, q_{n})*}E 
\otimes \mathcal{O}\left(
-p_{n}q_{n}H
\right)
\right) \\
=&\hom\left(
f^{(1, q_{n})*}E, 
\mathcal{O}\left(
p_{n}q_{n}f-vf
\right) 
\otimes \left(
\oplus L_{j}^{*\oplus \eta_{j}}
\right)[1]
\right), 
\end{split}
\]
where 
\[
L_{j}:=\mathcal{O}\left(
\frac{1}{q_{n}^2}\left(
-p_{n}q_{n}h+uh-K_{X}
+\sum a_{\rho}D_{\rho}
\right)
\right), \ 
0 \leq a_{\rho} \leq q_{n}^2-1. 
\]
Let 
$
M_{j}:=\mathcal{O}\left(
p_{n}q_{n}f-vf
\right) 
\otimes L_{j}^{*}, 
$
$H^{(n)}:=f^{(1, q_{n})*}H$. 
Assume that 
\[
\ch^{\beta_{n}H^{(n)}}_{1}(M_{j}) 
+2\frac{1}{q_{n}^2}H^{(n)}
\]
is anti-ample. 
Then 
\[
\begin{split}
H^{(n)2}.\ch^{\bar{\beta}H^{(n)}}_{1}(M_{j})
&=H^{(n)2}.\left(
\ch^{\beta_{n}H^{(n)}}_{1}(M_{j})
+(\beta_{n}-\bar{\beta})H^{(n)}
\right) \\ 
&<H^{(n)2}.\left(
\ch^{\beta_{n}H^{(n)}}_{1}(M_{j})
+\frac{1}{q_{n}^2}H^{(n)}
\right) \\ 
&<0
\end{split}
\]
and 
\[
\begin{split}
H^{(n)}.\ch^{\bar{\beta}H^{(n)}}_{1}(M_{j})^2 
&=H^{(n)}.\left(
\ch^{\beta_{n}H^{(n)}}_{1}(M_{j})
+(\beta_{n}-\bar{\beta})H^{(n)}
\right)^2 \\
&>H^{(n)}.\ch^{\beta_{n}H^{(n)}}_{1}(M_{j})
.\left(\ch^{\beta_{n}H^{(n)}}_{1}(M_{j}) 
+2\frac{1}{q_{n}^2}H^{(n)}
\right) \\ 
&>0. 
\end{split}
\]

Then the stability of $M_{j}$ and 
$f^{(1, q_{n})*}E$ shows that 
$\hom(f^{(1, q_{n}*)}E, M_{j}[1])=0$ 
as required. 
Hence it is enough to show that 
\[
\ch^{\beta_{n}H^{(n)}}_{1}(M_{j}) 
+2\frac{1}{q_{n}^2}H^{(n)}
\]
is anti-ample. 
We can compute it as 
\[
\begin{split}
\ch^{\beta_{n}H^{(n)}}_{1}(M_{j}) 
+2\frac{1}{q_{n}^2}H^{(n)}
&=-vf-\frac{uh-K_{X}+\sum a_{\rho}D_{\rho}}{q_{n}^2} 
+2\frac{1}{q_{n}^2}(h+q_{n}^2f) \\ 
&=-(v-2)f
-\frac{(u-2)h-K_{X}+\sum a_{\rho}D_{\rho}}{q_{n}^2}. 
\end{split}
\]
For $u, v>2$, this is anti-ample. 
\end{proof}

%===============================================================--
%==================================================================

\appendix
\section{Counter-example for Conjecture \ref{bg conj}}

In this appendix, we propose a counter-example for 
the original BG type inequality conjecture. 
More precisely, we show the following proposition 
using the argument of \cite{mar16}: 

\begin{prop}
\label{counterex}
Let $X$ be a Calabi-Yau threefold containing a plane 
$\mathbb{P}^2 \cong D \subset X$. Then there exists 
an ample divisor $H$ on $X$, $\alpha>0$, and 
$\beta \in \mathbb{R}$ such that 
$\mathcal{O}_{D}[1] \in \mathcal{A}_{\alpha, \beta}$ 
and 
\[
Z_{\alpha, \beta}\left(
\mathcal{O}_{D}[1]
\right) 
\in \mathbb{R}_{>0}. 
\]
\end{prop}

This proves the pair 
$(Z_{\alpha, \beta}, \mathcal{A}_{\alpha, \beta})$ 
is not a stability condition on $X$. 
In particular, Conjecture \ref{bg conj} does not 
hold by Theorem \ref{stab condi}. 

First we explain how to take the ample divisor $H$. 
Let $l \subset D$ be a line. Then 
\[
D^2=-3l, D^3=9. 
\]
Let $H'$ be any ample divisor on $X$ and put 
$L:=3H'+(H'.l)D$. 
Then $L$ is nef and big. 
Hence by the Kawamata-Shokurov basepoint-free theorem, 
some multiple of $L$ defines a birational morphism 
$f : X \to Y$, which only contracts $D$. 
In particular, there exists an ample $\mathbb{Q}$-divisor 
$A$ on $Y$ such that $L=f^{*}A$. 
On the other hand, $-D$ is $f$-ample since 
$\mathcal{O}_{D}(-3)=\omega_{D} \cong \mathcal{O}_{D}(D)$. 
Hence $H:=mL-\frac{1}{2}D$ is ample on $X$ for 
all $m \gg 0$. 

The Chern character of $\mathcal{O}_{D}$ is 
computed as : 
\[
H.\ch(\mathcal{O}_{D})
=\left(
0, \frac{9}{4}, \frac{9}{4}, \frac{3}{2}
\right). 
\]

Similarly, 
\[
H.\ch^{H}(\mathcal{O}_{D})
=\left(
0, \frac{9}{4}, 0, \frac{3}{8}
\right). 
\]

Before beginning the proof of Proposition \ref{counterex}, 
we recall the following aspect of 
the structure theorem of walls in tilt-stability: 

\begin{lem}
Let $v \in \Lambda:=\Image(H.\ch \colon K(X) \to \mathbb{Q}^4)$. 
Let $E \in \Coh^{\beta_{0}}(X)$ 
be an object with $H.\ch(E)=v$. 
Let $0 \to A \to E \to B \to 0$ be an exact sequence 
in $\Coh^{\beta_{0}}(X)$ 
with $H.\ch(E)=v$ 
which defines a wall $\mathcal{W}$ at 
$(\alpha_{0}, \beta_{0}) \in \mathcal{W}$. 
Then for every $(\alpha, \beta) \in \mathcal{W}$, 
we have $A, E, B \in \Coh^{\beta}(X)$. 
\end{lem}

\begin{proof}
See for example, 
Lemma 6.3 of \cite{abch13}. 
\end{proof}

Now we can prove Proposition \ref{counterex}: 

\begin{proof}[of Proposition \ref{counterex}]

The argument is exactly same as \cite{mar16}. 
Since 
\[
Z_{\alpha, 1}(\mathcal{O}_{D}[1])
=\frac{3}{8}(1-\alpha^2), 
\]
$Z_{\alpha, 1}(\mathcal{O}_{D}[1])>0$ if and only if 
$\alpha<1$. 
On the other hand, since 
$\nu_{\alpha, 1}(\mathcal{O}_{D})=0$, 
$\mathcal{O}_{D}[1] \in \mathcal{A}_{\alpha, 1}$ 
if $\mathcal{O}_{D}$ is $\nu_{\alpha, 1}$-semistable. 
Hence it is enough to show that 
there exists $\alpha \in (0, 1)$ such that 
$\mathcal{O}_{D}$ is $\nu_{\alpha, 1}$-stable. 

Let us consider the wall $\mathcal{W}$ 
which is defined by a short exact sequence 
\[
0 \to A \to \mathcal{O}_{D} \to B \to 0. 
\]
Let us denote 
$(r(A), c(A), d(A), e(A)):=H.\ch(A)$, etc. 
The center of this semicircular wall $\mathcal{W}$ is 
\[
\frac{d(\mathcal{O}_{D})}{c(\mathcal{O}_{D})}=1. 
\]
Let $R$ be the radius of the wall $\mathcal{W}$. 
We will bound $R$ from above. 
Since $\mathcal{H}^{-1}(B) \in \mathcal{F}_{\beta}$ and 
$A \in \mathcal{T}_{\beta}$ for all 
$(\alpha, \beta) \in \mathcal{W}$, 
we have 
\[
\frac{c(\mathcal{H}^{-1}(B))}{r(\mathcal{H}^{-1}(B))} 
\leq 1-R, 
\ \frac{c(A)}{r(A)} \geq 1+R.  
\]
By using the exact sequence 
\[
0 \to \mathcal{H}^{-1}(B) \to A \to \mathcal{O}_{D} \to 
\mathcal{H}^{0}(B) \to 0, 
\] 
we get 
\[
r(A)=r(\mathcal{H}^{-1}(B)), 
\ c(A) \leq c(\mathcal{H}^{-1}(B))+\frac{9}{4}. 
\]

Using these inequalities, we have 
\[
R \leq \frac{9}{8r(A)}
=\frac{9}{8m^3\ch_{0}(A)L^3-9\ch_{0}(A)}<1 
\]

for $m>1$. 
Since $\mathcal{O}_{D}$ is Gieseker stable, it is 
$\nu_{\alpha, \beta}$-stable for every 
$\alpha \gg 0, \beta \in \mathbb{R}$. 
By the bound of the radius of semicircular walls, 
 we conclude that $\mathcal{O}_{D}$ is 
$\nu_{\alpha, 1}$-stable for 
\[
\frac{9}{(8m^3L^3-9)}<\alpha<1. 
\]

\end{proof}

%===========================================================================
%===========================================================================

%\bibliographystyle{plain}
%\bibliography{bg-ineq.bib}

\end{document}